\documentclass[10pt]{article}
\font\smallit=cmti10
\font\smalltt=cmtt10

\usepackage{amssymb,latexsym,amsmath,epsfig,amsthm} 

\newcommand\numberthis{\addtocounter{equation}{1}\tag{\theequation}}

\makeatletter

\renewcommand\section{\@startsection {section}{1}{\z@}
{-30pt \@plus -1ex \@minus -.2ex}
{2.3ex \@plus.2ex}
{\normalfont\normalsize\bfseries}}

\renewcommand\subsection{\@startsection{subsection}{2}{\z@}
{-3.25ex\@plus -1ex \@minus -.2ex}
{1.5ex \@plus .2ex}
{\normalfont\normalsize\bfseries}}

\renewcommand{\@seccntformat}[1]{\csname the#1\endcsname. }

\makeatother

\newtheorem{theorem}{Theorem}

\newtheorem{corollary}{Corollary}

\begin{document}

\begin{center}
\uppercase{\bf Refinements of Some Partition Inequalities}
\vskip 20pt
{\bf James Mc Laughlin \footnote{This work was partially supported by a grant from the Simons Foundation (\#209175 to James Mc Laughlin).}}\\
{\smallit Department of Mathematics, 25 University Avenue,
West Chester University, West Chester, PA 19383}\\
{\tt jmclaughlin2@wcupa.edu}
\end{center}
\vskip 30pt
\centerline{\smallit Received: , Revised: , Accepted: , Published: } 
\vskip 30pt


\date{\today}

\centerline{\bf Abstract}
\noindent
In the present paper we initiate the study of a certain kind of partition inequality, by showing, for example, that if $M\geq 5$ is an integer and the integers $a$ and $b$ are relatively prime to $M$ and satisfy $1\leq a<b<M/2$, and the  $c(m,n)$ are defined by
\[
\frac{1}{(sq^a,sq^{M-a};q^M)_{\infty}}-\frac{1}{(sq^b,sq^{M-b};q^M)_{\infty}}:=\sum_{m,n\geq 0} c(m,n)s^m q^n,
\]
then $c(m, Mn)\geq 0$ for all integers $m\geq 0, n\geq 0$.

A similar result is proved for the integers $d(m,n)$ defined by
\[
(-sq^a,-sq^{M-a};q^M)_{\infty}-(-sq^b,-sq^{M-b};q^M)_{\infty}:=\sum_{m,n\geq 0} d(m,n)s^m q^n.
\]

In each case there are obvious interpretations in terms of integer partitions. For example, if $p_{1,5}(m,n)$ (respectively $p_{2,5}(m,n)$) denotes the number of partitions of $n$ into exactly $m$ parts $\equiv \pm 1 (\mod 5)$ (respectively $\equiv \pm 2 (\mod 5)$), then for each integer $n \geq 1$,
\[
p_{1,5}(m,5n)\geq p_{2,5}(m,5n), \,\,\,1 \leq m \leq 5n.
\]
\pagestyle{myheadings}
\markright{\smalltt INTEGERS: 16 (2016)\hfill}
\thispagestyle{empty}
\baselineskip=12.875pt
\vskip 30pt


\section{Introduction}

The purpose of this paper is to initiate the study of certain types of partition inequalities, which will be described below. Before coming to this refinement we recall some of the previous results.

Let  $p_{1,5}(n)$ (respectively $p_{2,5}(n)$) denote the number of partitions of $n$ into  parts $\equiv \pm 1 (\mod 5)$ (respectively $\equiv \pm 2 (\mod 5)$). As is well known,
\begin{align}\label{rrineqa}
\sum_{n=0}^{\infty}(p_{1,5}(n)-p_{2,5}(n))q^n &=\frac{1}{(q,q^4;q^5)_{\infty}}-\frac{1}{(q^2,q^3;q^5)_{\infty}}\\
&=\sum_{k=0}^{\infty}\frac{q^{k^2}}{(q;q)_k}-\sum_{k=0}^{\infty}\frac{q^{k^2+k}}{(q;q)_k}\notag\\
 &=\sum_{k=1}^{\infty}\frac{q^{k^2}(1-q^k)}{(q;q)_k}\notag\\
 &=\sum_{k=1}^{\infty}\frac{q^{k^2}}{(q;q)_{k-1}},\notag
\end{align}
where the second equality follows from the Rogers-Ramanujan identities. If the final series is expanded as a power series in $q$, the coefficient of $q^n$ is clearly non-negative for all $n$, and thus that
\begin{equation}\label{rrineq}
p_{1,5}(n)-p_{2,5}(n)\geq 0, \,\,\text{ for all } n \in \mathbb{N}.
\end{equation}

At the 1987 A.M.S. Institute on Theta Functions, Leon Ehrenpreis asked if \eqref{rrineq}
could be proved without employing the Rogers-Ramanujan identities. In \cite{K99}, Kadell showed that this was possible by finding an injection from the set of partitions counted by $p_{2,5}(n)$ into the set of partitions counted by $p_{1,5}(n)$. In \cite{BG05}, Berkovich and Garvan extended this result by giving injective proofs of an infinite family of partition inequalities implied by differences of finite $q$-products.
\begin{theorem}[Berkovich and Garvan, \cite{BG05}]\label{bgt}
Suppose $L > 0$, and $1 < r < m - 1$. Then the coefficients in the
$q$-expansion of the difference of the two finite products
\begin{equation}\label{bgeq}
\frac{1}{(q, q^{m-1}; q^m)_{L}} -
\frac{1}{(q^r, q^{m-r}; q^m)_{L}}
\end{equation}
are all nonnegative, if and only if $r \nmid (m - r)$ and $(m- r) \nmid r$.
\end{theorem}

In \cite{A13}, Andrews used his anti-telescoping
technique to provide an alternative answer to the question of Ehrenpreis. In the same paper he also proved two similar results for the non-negativity of the coefficients of the power series deriving from differences of finite $q$-products with modulus eight. One of these is contained in the following theorem (which Andrews called the ``finite little G\"ollnitz'' theorem).
\begin{theorem} [Andrews, \cite{A13}] If $L > 0$, and  the sequence $\{f_n\}$ is defined by
\[
\sum_{n=0}^{\infty}f_nq^n=\frac{1}{(q,q^{5},q^6;q^8)_{L}}-\frac{1}{(q^2,q^{3},q^7;q^8)_{L}},
\]
then $
f_n \geq 0, \text{ for all } n\geq 0$.
\end{theorem}

This result of Andrews was extended by Berkovich and  Grizzell, who proved combinatorially the following result (\cite[Theorem 1.3]{BG12}).
\begin{theorem}[Berkovich and Grizzell, \cite{BG12}]\label{bgt1} For any $L > 0$, and any odd $y > 1$, the $q$-series expansion of
\begin{equation}\label{bgeq1}
\frac{1}{(q,q^{y+2},q^{2y};q^{2y+2})_{L}}-\frac{1}{(q^2,q^{y},q^{2y+1};q^{2y+2})_{L}}
= \sum_{n=0}^{\infty}a(L,y,n)q^n
\end{equation}
has only non-negative coefficients.
\end{theorem}

They also proved an extension (Theorem 4.1 in \cite{BG12}) of the theorem above.
\begin{theorem} [Berkovich and Grizzell, \cite{BG12}] For any $L > 0$, and any odd $y > 1$, and any $x$ with $1 < x \leq y + 2$, the $q$-series expansion of
\begin{equation*}
\frac{1}{(q,q^{x},q^{2y};q^{2y+2})_{L}}-\frac{1}{(q^2,q^{y},q^{2y+1};q^{2y+2})_{L}}\\
= \sum_{n=0}^{\infty}a(L,x,y,n)q^n
\end{equation*}
has only non-negative coefficients.
\end{theorem}
The authors also give exact conditions under which the coefficients\\ $a(L,y,n)$ and $a(L,x,y,n)$ are equal to 0.
Berkovich and Grizzell continued their investigations in \cite{BG13}, where the following theorem is proved.
\begin{theorem} [Berkovich and Grizzell, \cite{BG13}] For any octuple of positive integers $(L, m, x, y, z, r, R, \rho)$,
 the $q$-series expansion of
\begin{multline*}
\frac{1}{(q^x,q^{y},q^z,q^{rx+Ry+\rho z};q^{m})_{L}}-\frac{1}{(q^{rx},q^{Ry},q^{\rho z},q^{x+y+z};q^{m})_{L}}\\
= \sum_{n=0}^{\infty}a(L,x,y,z, r,R,\rho,n)q^n
\end{multline*}
has only non-negative coefficients.
\end{theorem}
In each  of the above results, the finite $q$-products were all of the same order, and the modulus in each case was the same power of $q$. In \cite{BG14} they derived a result involving finite $q$-products of two different orders, and with two different moduli.
\begin{theorem} [Berkovich and Grizzell, \cite{BG14}] For any positive integers $m, n, y$, and $z$, with $\gcd(n, y) = 1$,
and integers $K$ and $L$, with $K \geq L \geq 0$,
\begin{equation*}
\frac{1}{(q^z; q^m)_K(q^{nyz}; q^{nm})_L}-\frac{1}{(q^{yz}; q^m)_K(q^{nz}; q^{nm})_L}
=\sum_{k=0}^{\infty}a(K,L,x,y,z, n,m,k)q^k
\end{equation*}
has only non-negative coefficients.
\end{theorem}
Note that Berkovich and Grizzell proved all of their results combinatorially, and that just as  the statement  at \eqref{rrineqa} was interpreted combinatorially at \eqref{rrineq}, each of the statements proved by those authors for differences of $q$-products may be interpreted  in terms of inequalities for certain restricted partition functions.

For example (employing the notation of the authors in \cite{BG12}), if
$P_1(L, y, n)$ denotes the number of partitions of $n$ into parts $\equiv 1, y + 2, 2y
(\mod (2y + 2))$ with the largest part less than  $(2y + 2)L$ and $P_2(L, y, n)$
denotes the number of partitions of $n$ into parts $\equiv 2, y, 2y +1 (\mod (2y +2))$ with
the largest part also less than  $(2y + 2)L$, then Theorem \ref{bgt1} implies that
\begin{equation*}
P_1(L, y, n)\geq P_2(L, y, n)
\end{equation*}
for all positive integers $L$ and $n$ (recall that $y$ is any odd integer greater than 1).

A refinement of the ordinary partition function $p(n)$ is $p(m,n)$, the number of partitions of $n$ into exactly $m$ parts, since
\[
p(1,n)+p(2,n)+\dots + p(n-1,n)+p(n,n) = p(n).
\]

This refinement  leads naturally  to a question that arises from the partition inequalities implied by the above theorems. Suppose $p_1(n)$ and $p_2(n)$ are two restricted partition counting functions such that
\[
p_1(n)\geq p_2(n), \,\,\text{ for all }  n \in \mathbb{N}.
\]
 Let $p_1(m,n)$ (respectively  $p_2(m,n)$) denote the number of partitions of the type counted by $p_1(n)$ (respectively  $p_2(n)$) into exactly $m$ parts. For which $n$ (if any) does it hold that
 \[
p_1(m,n)\geq p_2(m,n), \,\,  1\leq m \leq n?
\]

 Such questions are considered in the next section. An example of the results in the present paper is the following.

Let  $p_{1,5}(m,n)$ (respectively $p_{2,5}(m,n)$) denote the number of partitions of $n$ into exactly $m$ parts $\equiv \pm 1 (\mod 5)$ (respectively $\equiv \pm 2 (\mod 5)$). Then, for each integer $n \geq 1$,
\[
p_{1,5}(m,5n)\geq p_{2,5}(m,5n), \,\,\,1 \leq m \leq 5n.
\]

This is illustrated for $n=4$ (or $5n=20$) in Table \ref{Ta:t0} below.

\begin{center}
\begin{table}[ht]
\centering
  \begin{tabular} {|c|c|c ||c|c|c|}
    \hline
   $m$ &$p_{1,5}(m,20)$& $p_{2,5}(m,20)$ &   $m$ &$p_{1,5}(m,20)$& $p_{2,5}(m,20)$ \\ \hline

 1 & 0 & 0 & 11 & 1 & 0 \\
 2 & 4 & 4 & 12 & 2 & 0 \\
 3 & 0 & 0 & 13 & 0 & 0 \\
 4 & 5 & 5 & 14 & 1 & 0 \\
 5 & 4 & 3 & 15 & 1 & 0 \\
 6 & 2 & 2 & 16 & 0 & 0 \\
 7 & 4 & 3 & 17 & 1 & 0 \\
 8 & 1 & 1 & 18 & 0 & 0 \\
 9 & 2 & 1 & 19 & 0 & 0 \\
 10 & 2 & 1 & 20 & 1 & 0 \\
    \hline
  \end{tabular}
  \caption{$p_{1,5}(m,20)\geq p_{2,5}(m,20), \,\,\,1 \leq m \leq 20$.}\label{Ta:t0}
  \end{table}
\end{center}

\section{Main Results}
The example above follows as an implication of a special case of the next theorem.

\begin{theorem}\label{t1}
Let $M\geq 5$ be a positive integer, and let $a$ and $b$ be integers such that $1 \leq a< b <M/2$ and $\gcd(a,M)=\gcd(b,M)=1$. Define the integers $c(m,n)$ by
\begin{equation}\label{t1eq}
\frac{1}{(sq^a,sq^{M-a};q^M)_{\infty}}-\frac{1}{(sq^b,sq^{M-b};q^M)_{\infty}}:=\sum_{m,n\geq 0} c(m,n)s^m q^n.
\end{equation}

(i) Then $c(m, Mn)\geq 0$ for all integers $m, n\geq 0$.

(ii) If, in addition, $M$ is even, then $c(m, Mn+M/2)\geq 0$ for all integers $m, n\geq 0$.
\end{theorem}

\begin{proof}
We recall a special case of the $q$-binomial theorem:
\begin{equation}\label{btsc}
\sum_{n=0}^{\infty}\frac{z^n}{(q;q)_n}=\frac{1}{(z;q)_{\infty}}.
\end{equation}
Hence
\begin{multline}\label{ineqeq}
\frac{1}{(sq^a,sq^{M-a};q^M)_{\infty}}-\frac{1}{(sq^b,sq^{M-b};q^M)_{\infty}}\\
=\sum_{j,k\geq 0}\frac{s^{j+k}q^{a(j-k)+kM}}{(q^M;q^M)_j(q^M;q^M)_k}
-\sum_{j,k\geq 0}\frac{s^{j+k}q^{b(j-k)+kM}}{(q^M;q^M)_j(q^M;q^M)_k}.
\end{multline}
We fix the exponent of $s$ by setting $j+k=:m$, so that $j=m-k$ and the right side of \eqref{ineqeq} becomes
\begin{equation}\label{t1fm}
\sum_{m\geq 0}s^m\sum_{k=0}^m \frac{q^{a(m-2k)+kM}-q^{b(m-2k)+kM}}{(q^M;q^M)_{m-k}(q^M;q^M)_k}.
\end{equation}
Next, we restrict the values of $k$ so that when the inner sum is expanded as a power series, it contains only those powers of $q$ whose exponents are multiples of $M$ (so that the series multiplying $s^m$ is $\sum_{n=0}^{\infty}c(m,Mn)q^{Mn}$).

From the stated properties of $a$ and $b$, it can be seen that what is needed is the set of values of $k$ for which $m-2k$ is a multiple of $M$. If $m$ is even, then $k=m/2$ is such a value, and $q^{a(m-2k)+kM}-q^{b(m-2k)+kM}=0$ in this case. Hence we need only consider those $k$ in the intervals $0\leq k<m/2$ and $m/2<k\leq m$ satisfying $m-2k\equiv 0 (\mod M)$.

Next,  notice that every such $k'$ in the upper interval may be expressed as $k'=m-k$, for some $k$ in the lower interval, and every $k$ in the lower interval can be similarly matched with a $k'$ in the upper interval. Note that $m-2k\equiv 0 (\mod M)\Longleftrightarrow m-2(m-k)\equiv 0 (\mod M)$, and that the denominators of the summands remain invariant under the transformation $k\leftrightarrow m-k$. Hence
\begin{multline*}
\sum_{m,n\geq 0} c(m,Mn)s^m q^{Mn}\\
=\sum_{m\geq 0}s^m\sum_{\stackrel{0\leq k <m/2}{M|m-2k}} \frac{
\begin{matrix}q^{a(m-2k)+kM}-q^{b(m-2k)+kM}\phantom{asadasdasdaaaaaaaa}\\
\phantom{asaaaaaada}+q^{-a(m-2k)+(m-k)M}-q^{-b(m-2k)+(m-k)M}
\end{matrix}
}{(q^M;q^M)_{m-k}(q^M;q^M)_k}\\
=\sum_{m\geq 0}s^m\sum_{\stackrel{0\leq k <m/2}{M|m-2k}} \frac{
q^{a(m-2k)+kM}(1-q^{(m-2k)(b-a)})(1-q^{(m-2k)(M-b-a)})
}{(q^M;q^M)_{m-k}(q^M;q^M)_k}
\end{multline*}
Finally, $(m-2k)(b-a)$ and $(m-2k)(M-b-a)$ are each positive multiples of $M$ (since $M|m-2k$), and the conditions on $a$ and $b$ give that they are different multiples of $M$, each less than $(m-k)M$, so that the factors $(1-q^{(m-2k)(b-a)})$ and $(1-q^{(m-2k)(M-b-a)})$ are cancelled by two different factors in the $q$-product $(q^M;q^M)_{m-k}$. The remaining factors in the denominators may be expanded as geometric series with only non-negative coefficients, and the claim at (i) above follows.

The claim at (ii) follows similarly, upon noting that
\[
m-2k\equiv M/2 (\mod M)\Longleftrightarrow m-2(m-k)\equiv -M/2 \equiv M/2 (\mod M).
\]
\end{proof}

We next compare the results in Theorem \ref{t1} with the result in the Theorem \ref{bgt} of Berkovich and Garvan.   Our results are weaker than those of Berkovich and Garvan in Theorem \ref{bgt}, in the sense that setting $s=a=1$ in our theorem recovers only the case $L\to \infty$ in their theorem, and only in the the arithmetic progressions $0 (\mod M)$ and $M/2 (\mod M)$ (in the case $M$ is even). However, in the case of these arithmetic progressions, our result is stronger in two senses.

Firstly, the results hold for cases where $a>1$, in contrast to the result in Theorem \ref{bgt}, which holds only when $a=1$. Secondly, as we will see below, the inclusion of the parameter $s$ allows us to give stronger partition interpretations, in that the partition inequalities for integers in these arithmetic progressions also hold for any particular fixed number of parts.

\begin{corollary}\label{c1}
Let $M$, $a$ and $b$ be as in Theorem \ref{t1}. Let $p_{a, M, m}(n)$ denote the number of partitions of $n$ into exactly $m$ parts $\equiv \pm\, a (\mod M)$, and let  $p_{b,M,m}(n)$ likewise denote the number of partitions of $n$ into exactly $m$ parts $\equiv \pm\, b (\mod M)$. Then

 (i) $p_{a,M,m}(n M)\geq p_{b,M,m}(nM)$ for all integers $n\geq 1$, and all integers $m$, $1\leq m \leq Mn$.

 (ii) If $M$ is even, then $p_{a,M,m}(nM+ M/2)\geq p_{b,M,m}(nM+M/2)$ for all integers $n\geq 0$, and integers $m$ with $1\leq m \leq Mn+M/2$.
\end{corollary}
\begin{proof}
Clearly from \eqref{t1eq},  $p_{a,M,m}(n M)-p_{b,M,m}(nM)=c(m,Mn)$, so that (i) follows from Theorem \ref{t1}, and (ii) follows similarly.
\end{proof}

We next prove a companion result to that in Theorem \ref{t1}, one which has implications for the number of partitions into distinct parts.

\begin{theorem}\label{t2}
Let $M\geq 5$ be a positive integer, and let $a$ and $b$ be integers such that $1 \leq a< b <M/2$ and $\gcd(a,M)=\gcd(b,M)=1$. Define the integers $d(m,n)$ by \begin{equation}\label{t2eq}
(-sq^a,-sq^{M-a};q^M)_{\infty}-(-sq^b,-sq^{M-b};q^M)_{\infty}:=\sum_{m,n\geq 0} d(m,n)s^m q^n.
\end{equation}

(i) Then $d(m, Mn)\geq 0$ for all integers $m, n\geq 0$.

(ii) If, in addition, $M$ is even, then $d(m, Mn+M/2)\geq 0$ for all integers $m, n\geq 0$.
\end{theorem}

\begin{proof}
We begin by recalling another special case of the $q$-binomial theorem:
\begin{equation}\label{btsc2}
\sum_{n=0}^{\infty}\frac{a^n q^{n(n-1)/2}}{(q;q)_n}=(-a;q)_{\infty}.
\end{equation}
Hence
\begin{multline}\label{ineqeq2}
(-sq^a,-sq^{M-a};q^M)_{\infty}-(-sq^b,-sq^{M-b};q^M)_{\infty}\\
=\sum_{j,k\geq 0}\frac{s^{j+k}q^{a(j-k)+kM}q^{M[j(j-1)/2+k(k-1)/2]}}{(q^M;q^M)_j(q^M;q^M)_k}\\
-\sum_{j,k\geq 0}\frac{s^{j+k}q^{b(j-k)+kM}q^{M[j(j-1)/2+k(k-1)/2]}}{(q^M;q^M)_j(q^M;q^M)_k}
\end{multline}
We again fix the exponent of $s$ by setting $j+k=:m$, so that $j=m-k$ and the right side of \eqref{ineqeq2} becomes
\begin{equation}\label{t2fm}
\sum_{m\geq 0}s^m\sum_{k=0}^m \frac{(q^{a(m-2k)+kM}-q^{b(m-2k)+kM})q^{M[(m-k)(m-k-1)/2+k(k-1)/2]}}{(q^M;q^M)_{m-k}(q^M;q^M)_k}
\end{equation}

Note that the factor $(q^{a(m-2k)+kM}-q^{b(m-2k)+kM})$ in \eqref{t2fm} is the same as that in the numerator of \eqref{t1fm}, and that the rest of the summand in \eqref{t2fm} remains invariant under the transformation $k\longleftrightarrow m-k$. Hence the remainder of the proof parallels that of Theorem \ref{t1}, and so is omitted.
\end{proof}

The result in Theorem \ref{t2} may be interpreted in terms of certain restricted partitions into distinct parts.

\begin{corollary}\label{c2}
Let $M$, $a$ and $b$ be as in Theorem \ref{t2}. Let $p^*_{a, M, m}(n)$ denote the number of partitions of $n$ into exactly $m$ distinct parts $\equiv \pm\, a (\mod M)$, and let  $p^*_{b,M,m}(n)$ denote the number of partitions of $n$ into exactly $m$ distinct parts $\equiv \pm\, b (\mod M)$. Then

 (i) $p^*_{a,M,m}(n M)\geq p^*_{b,M,m}(nM)$ for all integers $n\geq 1$, and all integers $m$,  $1\leq m \leq Mn$.

 (ii) If $M$ is even, then $p^*_{a,M,m}(nM+ M/2)\geq p^*_{b,M,m}(nM+M/2)$ for all integers $n\geq 0$, and integers $m$ with $1\leq m \leq Mn+M/2$.
\end{corollary}
\begin{proof}
The proof is immediate from Theorem \ref{t2}, since from \eqref{t2eq},
\[
p^*_{a,M,m}(n M)-p^*_{b,M,m}(nM)=d(m,Mn),
\]
 so that (i) follows. The claim at (ii) follows similarly.
\end{proof}

\section{Concluding Remarks}

A number of obvious questions present themselves.

1. Are there combinatorial proofs of the inequalities in Corollaries \ref{c1} and \ref{c2}? More precisely, is there an injection from the partitions counted by $p_{b,M,m}(nM)$ to those counted by $p_{a,M,m}(nM)$, and an injection from the partitions counted by $p^*_{b,M,m}(nM)$ to those counted by $p^*_{a,M,m}(nM)$?

2. Are there any cases where ``finite'' versions of Theorems \ref{t1} and \ref{t2} hold, in the sense that if the ``$\infty$'' in the infinite products is replaced by a positive integer $L$ to give finite products, then all integers  $c(m,Mn)$ and $d(m,Mn)$ are still non-negative?

3. Theorem \ref{t1} may be thought of as a partial refinement/extension of Theorem \ref{bgt}. Do any of the other theorems in the introduction have similar partial refinement/extensions?

4. It seems that Theorem \ref{t1}  is not the end of the story for the type of infinite product difference shown on the left side of \eqref{t1eq}. For example, if
\begin{equation}\label{t1conj1}
\frac{1}{(sq^3,sq^{13};q^{16})_{\infty}}-\frac{1}{(sq^7,sq^{9};q^{16})_{\infty}}:=\sum_{m,n\geq 0} c(m,n)s^m q^n,
\end{equation}
then numerical evidence suggests that
\begin{align}
&c(m,16n+12)\geq 0,&& \,\, \text{ for all } n \geq 0,&& \,\, 1 \leq m \leq 16n+12,&\numberthis \label{conj11}\\
&c(m,16n+15)\geq 0,&& \,\, \text{ for all } n \geq 0,&& \,\, 1 \leq m \leq 16n+15.&\numberthis \label{conj12}
\end{align}

These observations motivate the following general problem.

Let $M\geq 5$ be a positive integer, and let $a$, $b$ and $r$ be integers such that $1 \leq a< b <M/2$ and $\gcd(a,M)=\gcd(b,M)=1$ and define the integers $c(a,b,M,m,n)$ by
\begin{equation*}
\frac{1}{(sq^a,sq^{M-a};q^M)_{\infty}}-\frac{1}{(sq^b,sq^{M-b};q^M)_{\infty}}:=\sum_{m,n\geq 0} c(a,b,M,m,n)s^m q^n.
\end{equation*}
Find all quadruples $(M,a,b,r)$ such that
\begin{equation}\label{cabeq}
 c(a,b,M,m, Mn+r)\geq 0, \hspace{25pt} \text{ for all } \,n\geq 0,\hspace{25pt}  \text{ for all } \, m\in [1, Mn+r].
\end{equation}

To help motivate further study of the problem, we list some such quadruples $(M, a, b ,r)$  in the  table below, values not given by Theorem \ref{t1} and for which experimental evidence suggests  \eqref{cabeq} holds.
\begin{table}[ht]
\centering
  \begin{tabular}{ |c | c | c|c |}
    \hline
   $M$ & $a$ & $b$ & $r$ \\ \hline
    12 & 1 & 5 &3, 4\\ \hline
    16 & 1 & 5 &4\\
     & 1 & 7 &3,4,6\\
    & 3 & 7 &12,15\\
    \hline
    18 & 1 & 7 &3,5,6\\\hline
    20 & 1 & 9 &3,4,5,6,8\\
    & 3 & 7 &4,15\\
    & 3 & 9 &1,12,15\\\hline
    24 & 1 & 5 &6\\
      & 1 & 7 &4,8,9\\
    & 1 & 11 &3,4,5,6,8,10\\
    & 5 & 11 &1,6,16,20,21\\
    & 7 & 11 &8,18\\
    \hline
  \end{tabular}
  \caption{Quadruples $(M, a, b ,r)$    not given by Theorem \ref{t1} and for which experimental evidence suggests  \eqref{cabeq} holds.}\label{Ta:t1}
  \end{table}

Note that $M$ is even for all values in the table.

Of course any set of values for  $M$, $a$, $b$ and $r$ for which \eqref{cabeq} holds
 also implies an infinite family of partition inequalities. Let $M$, $a$, $b$ and $r$  such that \eqref{cabeq} holds.  If $p_{a,M}(m,n)$ (respectively $p_{b,M}(m,n)$) denotes the number of partitions of the integer $n$ into exactly $m$ parts $\equiv \pm a (\mod M)$ (respectively $\equiv \pm b (\mod M)$), then for all positive integers $n$, and all  $m \in [1 , Mn+r]$,
\begin{equation}\label{partconj}
p_{a,M}(m,Mn+r)\geq p_{b,M}(m,Mn+r).
\end{equation}

5. Likewise, it seems that Theorem \ref{t2}  is not the end of the story either, for the type of infinite product difference shown on the left side of \eqref{t2eq}. Let $M\geq 8$ be an even positive integer, and let $a$ and $b$  be integers such that $1 \leq a< b <M/2$ and $\gcd(a,M)=\gcd(b,M)=1$. Define the polynomials $p_{a,b,M,n}(s)$ by
\begin{equation}\label{t1conj3}
(-sq^a,-sq^{M-a};q^M)_{\infty}-(-sq^b,-sq^{M-b};q^M)_{\infty}:=\sum_{n\geq 0} p_{a,b,M,n}(s) q^n.
\end{equation}
Experimental evidence appears to suggest that if $r$ is any fixed integer, $0\leq r \leq M-1$, then for all $n\geq 0$, all of the coefficients of $p_{a,b,M,Mn+r}(s)$ have the same sign.
As an example, $p_{1,3,8,n}(s)$ is shown in the following table for $320\leq n \leq 327$.

\begin{table}[ht]
\centering
  \begin{tabular}{|c|cc|}
  \hline
  n&& $p_{1,3,8,n}(s)$\\ \hline
 320 &&$ 3 s^8
   \left(s^2+5\right) $\\
 321 &&$ s \left(249 s^{10}+3872
   s^8+8355 s^6+3705 s^4+273
   s^2+1\right) $\\
 322 &&$ s^4 \left(10 s^8+614
   s^6+2367 s^4+1424
   s^2+127\right) $\\
 323 &&$ -s \left(161
   s^{10}+2775 s^8+6858
   s^6+3380 s^4+267
   s^2+1\right) $\\
 324 &&$ s^4 \left(2 s^8+141
   s^6+391 s^4+123
   s^2+3\right) $\\
 325 &&$ -s \left(266
   s^{10}+4010 s^8+8729
   s^6+3862 s^4+280
   s^2+1\right) $\\
 326 &&$ -s^4 \left(10 s^8+548
   s^6+2154 s^4+1375
   s^2+127\right) $\\
 327 &&$ s \left(228 s^{10}+3474
   s^8+7728 s^6+3582 s^4+273
   s^2+1\right) $\\ \hline
\end{tabular}
  \caption{The coefficients in $p_{1,3,8,n}(s)$ all have the same sign, $320\leq n \leq 327$.}\label{Ta:t2}
  \end{table}

This pattern of signs for all of the coefficients of $p_{1,3,8,n}(s)$ in Table \ref{Ta:t2}, namely $+,+,+,$ $-,+,-,-, +$, repeats modulo 8, as $n$ cycles through the various residue classes modulo 8 (this was checked up to $n=1920$). At this point we are unable to say if this pattern eventually breaks down for $n$ large enough, or if it holds for all even $M\geq 8$.  No similar patterns appear to hold for $M$ odd.

We leave it to others to hopefully cast further light on these questions.

 \allowdisplaybreaks{

}
\end{document}